\documentclass[12pt, twoside, leqno]{article}
 \usepackage{amsmath,amsthm}
 \usepackage{amssymb}
 
 \usepackage{enumerate}

 \usepackage{amsmath, color}
 \usepackage{cite}

 \newtheorem{theorem}{Theorem}[section]
 \newtheorem{corollary}[theorem]{Corollary}
 
 \newtheorem{proposition}[theorem]{Proposition}

 \newtheorem{acknowledgement}[theorem]{Acknowledgement}

 \theoremstyle{definition}
 \newtheorem{defn}[theorem]{Definition}
 
 \newtheorem{example}[theorem]{Example}
 
 \numberwithin{equation}{section}

\begin{document}

 	\baselineskip=17pt

 	\title{Super-wavelets on local fields of positive characteristic}

 	\author{
 		Niraj K. Shukla and Saurabh Chandra Maury
 		}
 	
 	\date{August 25, 2017}
 	
 	\maketitle

 	
 	\renewcommand{\thefootnote}{}
 	
 	
 	
 	\footnote{\emph{To be appear in}:  Mathematische Nachrichten}
 	
 	\renewcommand{\thefootnote}{\arabic{footnote}}
 	\setcounter{footnote}{0}
 	
 	
 	\begin{abstract}
 	The concept of super-wavelet   was  introduced by   Balan, and Han and Larson  over the field of real numbers which has many applications not only in engineering branches but also in different areas of mathematics.   To develop  this  notion  on local fields having  positive characteristic we obtain   characterizations of super-wavelets of finite length as well as  Parseval frame  multiwavelet  sets of finite order in this setup. Using the group theoretical approach based on coset representatives, further  we establish  Shannon type multiwavelet in this perspective  while providing   examples of Parseval frame (multi)wavelets and (Parseval frame) super-wavelets.  In addition, we  obtain necessary conditions for  decomposable and extendable Parseval frame wavelets  associated to Parseval frame super-wavelets.
 	\end{abstract}

 	\section{Introduction}
 	
 	Having applications in signal processing, data compression and image analysis, super-wavelets solve the problems of multiplexing in networking, which consists of sending multiple signals or streams of information on a carrier at the same time in the form of a single, complex signal and then recovering the separate signals at the receiving end. 
 	The concept of super-wavelets was   introduced by Balan in \cite{B1999}, Han and Larson in  \cite{HL2005}  as follows: \textit{A super-wavelet of length $n$ is an $n$-tuple $(f_1,f_2,...,f_n)$ in the direct sum Hilbert space $\displaystyle \bigoplus_{n} L^2(\mathbb R)$, such that the coordinated dilates of all its coordinated translates form an orthonormal basis for $\displaystyle \bigoplus_{n} L^2(\mathbb R)$}.   Here, every $f_i$ is known as a \textit{component} of the super-wavelet.

 	Our main goal is to  develop the theory of super-wavelets in the setting of local fields having positive characteristic while for the local field, Jiang, Li and Jin in \cite{JLJ2004} introduced the concepts of multiresolution analysis (MRA) in which the ring of integers plays an important role. By the \textit{local field}, we mean a finite characteristic field which is locally compact, non-discrete, and totally disconnected. Actually, such fields (for example:  Cantor dyadic group,  Vilenkin $p$-groups) have a   formal power series over a finite fields $GF(p^c)$. If $c=1$, it is a $p$-series field while for $c\neq 1$, it is an algebraic extension of degree $c$ of a $p$-series field. As an  application point of view, such fields are very much useful in computer science, cryptographic protocols, etc.
 	
 	The notion of orthonormal multiwavelets,     multiwavelet sets, Parseval frame multiwavelets and  Parseval frame multiwavelet sets have been extensively studied by many authors for one dimensional as well as  higher dimensional Euclidean spaces \cite{BRS2001,DLS1997,D2005}, and further, these are developed in different perspectives,   namely,  locally compact abelian groups, local fields, $p$-adic fields $\mathbb Q_p$,        Vilenkin $p$-groups,   etc.  \cite{KSS2009-1, AKS2011,AK,AS2012,BJ2012,BJ2015,B2004,BB2004,F2008,HLPS1999,L1996, SV2015, SMM2017}  by a large number of researchers.   
 	
 	Dahlke   introduced the concept of wavelets  in locally compact abelian groups \cite{D1993} while it was  generalized to abstract Hilbert spaces by   Han,   Larson,   Papadakis and   Stavropoulos \cite{HLPS1999}. Further,        Benedetto and   Benedetto  developed a wavelet theory for local fields and related groups in \cite{B2004,BB2004}.  At this juncture, it is pertinent to mention that  Khrennikov with his collaborators  in \cite{KSW2013}          introduced new ideas  to construct various infinite-dimensional multiresolution analyses (MRAs) and further for an application point of view, they  developed the theory of pseudo-differential operators and equations over the ring of adeles as well. A rigorous study of wavelets on $p$-adic field $\mathbb Q_p$  and its related property has been done by many authors including Albeverio, Khrennikov and Skopina \cite{AKS2011,  AKS2010, AK,AS2012,KSW2013,KSS2009,KSS2009-1}.

 	During the development of super-wavelets for the local fields, we obtain a characterization of Parseval frame multiwavelet  sets of finite order in Section 3 that also characterizes all multiwavelet sets. In the same section, we provide Shannon type multiwavelet along with some other examples of Parseval frame    (multi)wavelet sets which are associated with multiresolution analysis. Further, in Section 4,  we obtain two characterizations in which  one for super-wavelets,  and other for  super-wavelets whose each components are minimally supported, while  providing    examples  of   super-wavelets of   length $n$.   In the last section, the decomposable frame wavelets   and their properties are discussed. A rigorous study of super-wavelets and decomposable frame wavelets for the Euclidean spaces has been done by many authors in the references \cite{BDP2005,DDG2004,DL2011,DJ2007,GH2005,HL2005,ZX2012}.

 	In the next section, we  give a brief introduction about local fields. More details about the same can be seen in a   book by Taibleson \cite{T1975}.

 	
 	\section{Preliminaries on  local fields}
 	Throughout the paper, $K$ denotes a local field. By a local field we mean a field which is locally compact, non-discrete, and totally disconnected.   The set 
 	$$
 	\mathcal O=\left\{x\in K:\ |x|\leq 1\right\}
 	$$ 
 	denotes the ring of integers which is a unique maximal compact open subring of $K$, where the absolute value  $|x|$ of $x \in K$ satisfies the properties (for more details, we refer  \cite{T1975}):
 	\begin{enumerate}
 		\item[(i)] $  |x|=0$ if and only if $x=0$
 		\item[(ii)]  $  |xy|=|x||y|$, and   
 		\item[(iii)]	$ \ |x+y|\leq \mbox{max}\left\{{|x|,\; |y|}\right\}$, for all $x,\ y\in K$. The equality holds in case of $|x|\neq |y|$.  
 	\end{enumerate}
 	Further, we   consider   a maximal and prime ideal 
 	$$
 	\frak P=\left\{{x\in K:\ |x|<1}\right\}
 	$$
 	in $\mathcal O$, then $\mathfrak {P}=\mathfrak {p}\mathcal O,$ for an element $\mathfrak p$ (known as \textit{prime element}) of $\mathfrak P$ having maximum absolute value in view of totally disconnectedness of $K$, and hence, $\mathfrak P$ is compact and open. Therefore, the residue space $\mathcal Q= {\mathcal O}/{\mathfrak {P}}$ is isomorphic to a finite field $GF(q)$, where $q=p^c$ for some prime $p$ and positive  integer $c.$

 	For a measurable subset $E$ of $K$, let 
 	$$
 	|E|=\int_{K}\chi_{E}(x)dx,
 	$$
 	where $\chi_{E}$ is the characteristic function of $E$ and $dx$ is the Haar measure for $K^+$  (locally compact additive group of K), so $|\mathcal O|=1$. By decomposing $\mathcal O$ into $q$ cosets of $\mathfrak P$, we have $|\mathfrak{P}|=q^{-1}$ and $|\mathfrak{p}|=q^{-1}$, and hence for $x\in K\backslash \{0\}=:K^*$ (locally compact multiplicative group of K), we have $|x|=q^k$, for some $k \in \mathbb Z.$ Further, notice that $\mathcal O^*:=\mathcal O\backslash \frak P$ is the group of units in $K^*$, and for  $x\neq 0$,   we may write $x=\frak p^k x'$ with $x'\in \mathcal O^*$. In the sequel, we denote $\mathfrak{p}^k \mathcal {O}$ by $\mathfrak P^k$, for each $k \in \mathbb Z$ that is known as \textit{fractional ideal}. Here, for    $x \in \mathfrak{P}^k$, $x$ can be expressed uniquely as $$
 	x=\sum^{\infty}_{l=k}c_l\ \frak p^l,\ c_l\in \frak U, \ \mbox{and} \  c_{k}\neq 0,
 	$$ 
 	where  $\frak U=\left\{c_i\right\}_{i=0}^{q-1}$ is  a fixed full set of coset representatives of   $\frak P$ in $\mathcal O$.
 	
 	Let $\chi$ be a fixed character on $K^+$ that is trivial on $\mathcal {O}$ but is nontrivial on $\frak P^{-1}$, which  can be  found   by starting with nontrivial character and rescaling.  For $y\in K$, we define 
 	$$
 	\chi_y(x)=\chi(yx),   x\in K.
 	$$
 	For $f\in L^1(K)$,  the \textit{Fourier transform} of $f$ is the function $\hat{f}$ defined by
 	$$\hat{f}(\xi)=\int_{K}f(x)\overline{{\chi_\xi(x)}}dx=\int_{K}f(x)\chi(-\xi x)dx,$$
 	which can be extended for $L^2(K)$. 
 	
 	Notation $\mathbb N_0:=\mathbb N\cup \left\{0\right\}$. Let $\chi_u$ be any character on $K^+$. Since $\mathcal {O}$ is a subgroup of $K^+$, it follows that the restriction $\chi_{u|_{\mathcal {O}}}$ is a character on $\mathcal {O}$. Also, as a character on $\mathcal {O}$, we have $\chi_u=\chi_v$ if and only if $u-v\in \mathcal {O}$.  Hence, we have the following result \cite[Proposition 6.1]{T1975}:
 	
 	\begin{theorem} \label{Thm2.1}
 		Let $\mathcal Z:=\left\{u(n)\right\}_{n\in \mathbb N_{0}}$ be a complete list of (distinct) coset representation of $\mathcal {O}$ in $K^+$. Then, the set  
 		$$
 		\left\{\chi_{u(n)|_{\mathcal O}}\equiv \chi_{u(n)}\right\}_{n\in \mathbb N_0}
 		$$
 		is a list of (distinct) characters on $\mathcal {O}$. Moreover, it is a complete orthonormal system on $\mathcal {O}$. 
 	\end{theorem}

 	Next, we proceed to impose a natural order on $\mathcal Z$ which is used to develop the theory of Fourier series on $L^2 (\mathcal O)$. For this,  we choose a set $\left\{1=\epsilon_0,  \epsilon_i\right\}_{i=1}^{c-1}\subset \mathcal O^*$ such that the vector space $\mathcal Q$ generated by $\left\{1=\epsilon_0,\epsilon_i\right\}_{i=1}^{c-1}$ is isomorphic to the vector space $GF(q)$ over  finite field $GF(p)$ of order $p$  as   $q=p^c$. For $n\in \mathbb N_{0}$ such that $0\leq n<q,$ we write 
 	$$
 	n=\sum_{k=0}^{c-1} a_{k}p^{k},
 	$$ 
 	where  $0\leq a_{k}<p$. By noting that $\left\{u(n)\right\}_{n=0}^{q-1}$ as a complete set of coset representatives of $\mathcal O$ in $\frak P^{-1}$ with $|u(n)|=q$, for $0< n <q$ and $u(0)=0$, we define  
 	$$
 	u(n)=(\sum_{k=0}^{c-1} a_k\epsilon_k) \mathfrak{p}^{-1}.
 	$$ 
 	Now, for $n\geq 0$, we write $n=\sum_{k=0}^s b_{k}q^{k}$, where $0\leq b_{k}<q$, and define 
 	$$
 	u(n)= \sum_{k=0}^s u(b_k)\frak {p}^{-k}.
 	$$ 
 	In general, it is not true that $u(m+n)= u(m)+u(n)$ for each non-negative $m, n$ but
 	\begin{align*}
 	u(rq^k+s)=  u(r)\frak p^{-k}+u(s), \  \mbox{if}\  r\geq 0, \ k\geq 0,  \ \mbox{and} \  0\leq s< q^{k}.
 	\end{align*}

 	Now, we sum up above in the following theorem (see,  \cite[Proposition 6.6]{T1975},  \cite{BJ2012}):

 	\begin{theorem} \label{Thm2.2}  
 		For $n\in \mathbb N_{0}$, let $u(n)$ be defined as above. Then, we have  
 		\begin{itemize}
 			\item [(a)]  $u(n)=0$ if and only if $n=0$. If $k\geq 1$, then we have $|u(n)|=q^{k}$ if and only if $q^{k-1}\leq n<q^{k}.$ 
 			
 			\item [(b)]  $\left\{u(k):\ k\in \mathbb N_{0}\right\}=\left\{-u(k):\ k\in \mathbb N_{0}\right\}$.
 			
 			\item [(c)]   For a fixed $l\in \mathbb N_{0}$, we have $\left\{u(l)+u(k):\ k\in \mathbb N_{0}\right\}=\left\{u(k):\ k\in \mathbb N_{0}\right\}$. 
 		\end{itemize}
 	\end{theorem}
 	\noindent Following result and definition  will be used in the sequel [18]:
 	
 	\begin{theorem} \label{Thm2.3}
 		For all $l,\ k\in \mathbb N_{0},\ \chi_{u(k)}(u(l))=1.$
 	\end{theorem}
 	
 	\begin{defn}\label{Def2.4}
 		A function $f$ defined on $K$ is said to be  \textit{integral periodic} if $$f(x+u(l))=f(x), \ \  \textrm{for \ all} \ l\in \mathbb N_{0}, x\in K.$$
 	\end{defn}
 	\section{Parseval frame multiwavelet sets  for local fields}

 	Let $K$ be a local field of characteristic $p>0,\ \frak p$ be a prime element of $K$ and $u(n)\in K$ for $n\in \mathbb N_{0}$ be defined as above. Then a finite set $\Psi=\left\{\psi_m:\ m=1,\:2,...,M\right\}\subset L^2(K)$ is called a \textit{Parseval frame multiwavelet} of order $M$ in $L^2(K)$ if the system
 	$$
 	\mathcal A(\Psi):=\left\{\psi_{m, j, k}:=D^jT^{k}\psi_m:\ 1\leq m\leq M, \ j\in \mathbb Z,\ k\in \mathbb N_{0} \right\}
 	$$
 	forms a  Parseval frame   for $L^2(K)$, that means, for each $f \in L^2(K)$,
 	\begin{align*}
 	\hspace{3cm} 	\|f\|^2= \sum_{m=1}^M\sum_{(j, k) \in \mathbb Z \times \mathbb N_0}\left|<f, D^jT^{k}\psi_m>\right|^2,
 	\end{align*}
 	where the dilation and translation operators are defined as follows:
 	$$
 	D^j f(x)=q^{j / 2}f(\frak{p}^{-j}x), \ \mbox{and}\ T^k f(x)=f(x-u(k)), \ \ x \in K.
 	$$
 	If the system $\mathcal A(\Psi)$ is an orthonormal basis for $L^2(K)$, $\Psi$ is called an \textit{orthonormal multiwavelet} (simply,  \textit{multiwavelet}) of order $M$  in $L^2(K)$. In the case of Parseval frame system $\mathcal A(\{\psi\})$ for $L^2(K)$, $\psi$ is known as \textit{Parseval frame  wavelet}. Moreover, a Parseval frame multiwavelet $\Psi$ is known as \textit{semi-orthogonal} if $D^j  W \bot D^{j'}  W $,  for $j \neq j'$, where $W=\overline{\mbox{span}}\{T_k \psi: k \in \mathbb N_0, \psi \in \Psi\}$.

 	Notice that for $f \in L^2(K)$ and $\xi \in K$, we have
 	$$
 	\widehat{\left(D^jT^k f\right)}(\xi)=q^{-j/2} \chi_{u(k)}(-\mathfrak p^j \xi) \widehat{f} (\mathfrak p^j \xi), \ \ \mbox{for} \ j \in \mathbb Z, \ k \in \mathbb N_0.
 	$$
 	The following is a necessary and sufficient condition for the system $\mathcal A(\Psi)$ to be a Parseval frame   for $L^2(K)$ \cite{BJ2012}:
 	
 	\begin{theorem} \label{Thm3.1}
 		Suppose $\Psi=\left\{\psi_m:\ m=1,\:2,...,M\right\}\subset L^2(K).$ Then the affine system $\mathcal A(\Psi)$ is a Parseval frame   for $L^2(K)$
 		if and only if for a.e. $\xi,$ the following holds:
 		\begin{enumerate}
 			\item [(i)]  $\displaystyle\sum_{m=1}^M\sum_{j\in\mathbb Z} \left|\widehat{\psi}_m(\mathfrak p^{-j} \xi)\right|^2=1,$ \hfill  (3.1) 
 			\item [(ii)]  $\displaystyle \sum_{m=1}^M\sum_{j\in\mathbb N_0} \widehat{\psi}_m(\mathfrak p^{-j} \xi)\overline{\widehat{\psi}_m(\mathfrak p^{-j} (\xi+u(s))}=0,$ \ \ for   $s \in  \mathbb N_0\backslash q \mathbb N_0.$ \hfill (3.2)
 		\end{enumerate}
 	\end{theorem}

 	In particular, $\Psi$ is a multiwavelet in $L^2(K)$ if and only if $\|\psi_m\|=1$,  for $1 \leq m \leq  M$,  and the above   conditions   (3.1) and (3.2)  hold.

 	In the sequel of development of wavelets associated with an MRA on local fields of positive characteristics,   Jiang,  Li and  Jin in \cite{JLJ2004}  obtained    a  necessary and sufficient condition  for the system   $\left\{\varphi(\cdot-u(k)): k \in \mathbb N_0\right\}$   to constitute an orthonormal system  which is   as follows:
 	\begin{align*}
 	\hspace{4cm}	\sum_{k\in \mathbb N_0}\left|\widehat{\varphi}(\xi+u(k))\right|^2=1, \ \ a.e. \ \xi,
 	\end{align*}
 	for any $\varphi \in L^2(K)$.
 	
 	Notice that for all $\xi\in K$,   $0 \leq \displaystyle\sum_{k\in \mathbb N_0}\left|\widehat{\varphi}(\xi+u(k))\right|^2\leq 1$ if $\widehat{\varphi}=\chi_{\frak p \mathcal O},$ since $\frak p \mathcal O \subset \mathcal O$,  and the system $\{\mathcal O +u(k): k \in \mathbb N_0\}$ is a measurable partition of $K$. 
 	The following is a   generalization of above characterization:
 	
 	\begin{theorem}\label{Thm3.2}  
 		Let  $\varphi \in L^2(K)$. Then a necessary and sufficient condition for the system $\left\{\varphi(\cdot-u(k)): k \in \mathbb N_0\right\}$ to be a Parseval   frame  for $\overline{\mbox{span}}\{\varphi(\cdot-u(k)): k \in \mathbb N_0\}$   is as follows:
 		\begin{align*}
 		\hspace{4cm} 0 \leq \sum_{k\in \mathbb N_0}\left|\widehat{\varphi} (\xi+u(k))\right|^2\leq 1, \ \ a.e. \ \xi.
 		\end{align*}
 	\end{theorem}
 	
 	\begin{proof}  
 		Notice that  for every $f \in \overline{\mbox{span}}\left\{\varphi(\cdot-u(k)): k \in \mathbb N_0\right\}=:V_\varphi$, we have $\widehat{f}(\xi)=r(\xi) \widehat{\varphi}(\xi),$
 		for some integral periodic function $r \in L^2\left(\mathcal O, w\right)$,  where $w(\xi)=\sum_{k\in \mathbb N_0}\left|\widehat{\varphi}\left(\xi+u(k)\right)\right|^2, $
 		and hence
 		\begin{align*}
 		\sum_{k \in \mathbb N_0} \left|<f, T^k\varphi>\right|^2=&\sum_{k \in \mathbb N_0} \left|\int_K \widehat{f}(\xi) \overline{\widehat{\varphi}(\xi)} \chi_{u(k)}(\xi) d\xi\right|^2\\
 		=&\sum_{k \in \mathbb N_0} \left|\sum_{l\in \mathbb N_0}\int_\mathcal O \widehat{f}(\xi+u(l)) \overline{\widehat{\varphi}(\xi+u(l))} \chi_{u(k)}(\xi+u(l)) d\xi\right|^2\\
 		=&\sum_{k \in \mathbb N_0} \left|\int_\mathcal O \left(\sum_{l\in \mathbb N_0} r(\xi+u(l)) | \widehat{\varphi}(\xi+u(l))|^2 \right)    \chi_{u(k)}(\xi) d\xi\right|^2,
 		\end{align*}
 		since the system $\{\mathcal O +u(k): k \in \mathbb N_0\}$ is a measurable partition of $K$, and for all  $l,\ k\in \mathbb N_{0},\ \chi_{u(k)}(u(l))=1$ in view of Theorem \ref{Thm2.3}. Further, as the function $r$ is integral periodic, we write the above expression as follows:
 		\begin{align*}
 		\sum_{k \in \mathbb N_0} \left|<f, T^k\varphi>\right|^2=& \sum_{k \in \mathbb N_0} \left|\int_{\mathcal O} r(\xi)w(\xi)\chi_{u(k)}(\xi)d\xi\right|^2 =  \int_{\mathcal O} |r(\xi)|^2|w(\xi)|^2 d\xi,
 		\end{align*}
 		because of Theorem \ref{Thm2.1}. Therefore,  we have   condition 
 		$$ 
 		\int_{\mathcal O} |r(\xi)|^2|w(\xi)| d\xi= \int_{\mathcal O} |r(\xi)|^2|w(\xi)|^2 d\xi, 
 		$$
 		since  for every $f \in V_\varphi$, we have $\|f\|^2 =  \int_{\mathcal O} |r(\xi)|^2|w(\xi)| d\xi. $
 		That means, 
 		$$  
 		\int_{\mathcal O} |r(\xi)|^2w(\xi)\left(\chi_{\Omega}(\xi)- w(\xi)\right) d\xi=0, 
 		$$
 		holds  for all  integral periodic functions $r \in L^2\left(\mathcal O, w\right)$  if and only if $w(\xi)= \chi_{\Omega}(\xi),  \ a.e. \ \xi,$ where $$\Omega=\mbox{supp}\,  w \equiv \{\xi \in K :  w (\xi)\neq 0\}.$$   Now, it is enough to show that $f \in V_\varphi$ if and only if 
 		$$
 		\widehat{f}(\xi)=r(\xi)\widehat{\varphi}(\xi),
 		$$
 		for some integral periodic function $r \in L^2\left(\mathcal O, w\right)$.
 		This follows by noting  that $V_\varphi=\overline{\mathcal A_\varphi}$, $L^2\left(\mathcal O, w\right)=\overline{\mathcal P_\varphi}$ and the operator $U: \mathcal A_\varphi \rightarrow \mathcal P_\varphi$ defined by $U(f)(\xi)=r(\xi)$ is an isometry which is onto, where  
 		$$
 		\mathcal A_\varphi=\mbox{span}\left\{T^k\varphi: k \in \mathbb N_0\right\},
 		$$
 		and  $\mathcal P_\varphi$ is the space of all integral periodic trigonometric polynomials $r$ with the $L^2\left(\mathcal O, w\right)$ norm
 		$$
 		\|r\|^2_{L^2\left(\mathcal O, w\right)}=\int_{\mathcal O} |r(\xi)|^2 w(\xi) d\xi.
 		$$
 		Here, $f \in \mathcal A_\varphi$ if and only if  for  $r\in \mathcal P_\varphi$, $\widehat{f}(\xi)=r(\xi)\widehat{\varphi}(\xi),$ where $$r(\xi)=\displaystyle\sum_{k\in \mathbb N_0} a_{k} \overline{\chi_{u(k)}(\xi)},$$ for  a finite number of non-zero  elements of $\{a_{k}\}_{k \in \mathbb N_0}$.  Now, by  {splitting} the integral into cosets of $\mathcal O$ in $K$ and using the fact of integral periodicity of  $r$, we have
 		\begin{align*}
 		\|f\|^2_2=\int_{\mathcal O} \sum_{k\in \mathbb N_0} \left|\widehat{f}\left(\xi+u(k)\right)\right|^2 d\xi
 		=\int_{\mathcal O} |r(\xi)|^2\sum_{k\in \mathbb N_0} \left|\widehat{\varphi}\left(\xi+u(k)\right)\right|^2 d\xi
 		=\|r\|^2_{L^2\left(\mathcal O, w\right)},
 		\end{align*}
 		which shows that the operator $U$ is an isometry.
 	\end{proof}

 	Following result gives a characterization of bandlimited   Parseval frame multiwavelets in $L^2(K)$:

 	\begin{theorem}\label{Thm3.3}  
 		Let $\Psi=\{\psi_m\}_{m=1}^M \subset  L^2(K)$ be such that for each $m\in \{1, 2, \cdots, M\}$, $|\widehat{\psi}_m|=\chi_{W_m},$ and   $W=\bigcup_{m=1}^M W_m$ is a disjoint union of measurable subsets of $K$.  Then  $\Psi$ is a  semi-orthogonal Parseval frame multiwavelet in $L^2(K)$ if and only if following hold:
 		\begin{enumerate}
 			\item [(i)]  \textit{$\{\mathfrak p^{j} W:  j \in \mathbb Z\}$ is a measurable partition of $K$, and}
 			\item [(ii)] \textit{for each $m\in \{1, 2, \cdots, M\}$, the set $\{W_m+u(k):k \in \mathbb N_0\}$ is a measurable partition of   a subset of $K$}.
 		\end{enumerate}
 	\end{theorem}
 	Such  set $W$ is known as \textit{Parseval frame multiwavelet set} (of order $M$) in $K$.
 	
 	\begin{proof}   
 		Let $\Psi=\{\psi_m\}_{m=1}^M\subset L^2(K)$ be such that $|\widehat{\psi}_m|=\chi_{W_m},$ where  $W=\bigcup_{m=1}^M W_m$ is a measurable subset of $K$. Then,   the condition (3.1) of Theorem \ref{Thm3.1}  yields that $   \bigcup_{j\in\mathbb Z} \mathfrak p^{j} W=K$, a.e.,  that is  equivalent to the part (i), which also gives that for $j \geq 0$, $|\frak p^{j} W_m \cap W_{m'}|=0$, for each $m, m'\in \{1, 2, \cdots, M\}$, and $m \neq m'$. Further in view of Theorem \ref{Thm3.2},  the system   $$
 		\left\{\psi_m(\cdot-u(k)): k \in \mathbb N_0\right\}, \  m\in \{1, 2, \cdots, M\}
 		$$   is a  Parseval   frame  for $\overline{\mbox{span}}\left\{\psi_m(\cdot-u(k)): k \in \mathbb N_0\right\}$    in $L^2(K)$ if and only if 
 		$$
 		\sum_{k\in \mathbb N_0}\left|\widehat{\psi}_m (\xi+u(k))\right|^2=\sum_{k\in \mathbb N_0}\chi_{W_m} (\xi+u(k))\leq 1, \ a.e. \ \xi,
 		$$       that is  equivalent to the part (ii). In this case  
 		$$ \{f \in L^2(K): \mbox{supp}\widehat{f} \subset W\}=\overline{\mbox{span}}\{\psi(\cdot-u(k)): \psi \in \Psi, k \in \mathbb N_0\}=:W_0.
 		$$ 
 		By scaling $W_0$ for any $j \in \mathbb Z$, we have
 		$$
 		D^jW_0=\overline{\mbox{span}}\{D^j\psi(\cdot-u(k)): \psi \in \Psi, k \in \mathbb N_0\}= \{f \in L^2(K): \mbox{supp}\widehat{f} \subset \frak p^{-j} W\}.
 		$$
 		Therefore, $\Psi$ is a semi-orthogonal Parseval frame multiwavelet in $L^2(K)$   if and only if   $\bigoplus_{ j \in \mathbb Z} D^jW_0=L^2(K)$ and (ii) hold, which is true  if and only if (i) and (ii) hold. 
 	\end{proof}

 	\begin{corollary} \label{Cor3.4}  
 		Let $\Psi=\{\psi_m\}_{m=1}^M   \subset  L^2(K)$ be such that for each $m\in \{1, 2, \cdots, M\}$, $|\widehat{\psi}_m|=\chi_{W_m},$ and   $W=\bigcup_{m=1}^M W_m$ is a disjoint union of measurable subsets of $K$.  Then  $\Psi$ is a    multiwavelet in $L^2(K)$ if and only if the following hold:
 		\begin{enumerate}
 			\item [(i)]  \textit{$\{\mathfrak p^{j} W:  j \in \mathbb Z\}$ is a measurable partition of $K$, and}
 			\item [(ii)] \textit{for each $m\in \{1, 2, \cdots, M\}$, the system $\{W_m+u(k):k \in \mathbb N_0\}$ is a measurable partition of     $K$}.
 		\end{enumerate}
 	\end{corollary}
 	Such set $W$ is known as \textit{multiwavelet set} (of order $M$) in $K$.
 	
 	The most elegant method to construct    multiwavelets is based on       multiresolution analysis (MRA) which is a family of closed subspaces of a Hilbert space satisfying certain properties.  By an  \textit{MRA}, we mean that   a sequence of closed subspaces $\left\{V_j\right\}_{j\in \mathbb Z}$ of $L^2(K)$ satisfying the following properties:  for all $j\in \mathbb Z$,
 	\begin{enumerate}
 		\item [(i)] $V_j\subset V_{j+1}$,\;\;    $D V_j =V_{j+1}$, \;\; $\overline{\bigcup_{j\in \mathbb Z}V_j} =L^2(K)$,\;\;
 		$\bigcap_{j\in \mathbb Z}V_j=\left\{0\right\}$, and
 		
 		\item [(ii)] there is a  $\varphi\in V_0$ (known as, \textit{scaling function})  such that $\{ \varphi(\cdot-u(k))\}_{k\in \mathbb N_{0}}$ forms an   orthonormal basis    for $V_{0}$.
 	\end{enumerate}
 	If we replace the term \lq\lq orthonormal basis \rq \rq by \lq \lq Parseval frame \rq \rq in the last axiom, then above is known as \textit{Parseval frame MRA}.

 	Now, we provide an example of multiwavelet set associated with an MRA with the help of   ring of integers:
 	
 	\begin{example}
 		[\textbf{Shannon type  Multiwavelet}]\label{Exa3.5}  Let us consider the ring of integers $\mathcal O$ in $K$. Then, $\mathcal O$ is an additive subgroup  of $\frak P^{-1}$,  and hence the system 
 		$$
 		\left\{\mathcal O+u(0), \mathcal O+u(1), \cdots, \mathcal O+u(q-1)\right\}
 		$$
 		is a measurable partition of $\frak P^{-1}$, where  the set $\left\{u(n)\right\}_{n=0}^{q-1}$ is  a complete set of distinct coset representatives  of $\mathcal O$ in $\frak P^{-1}$ with $u(0)=0$, and $|u(n)|=q$, for $0< n <q$. Thus the system  
 		$$  
 		\{\mathcal O+u(1), \cdots, \mathcal O+u(q-1)\} 
 		$$ is a measurable partition of the set $\frak P^{-1} \backslash \mathcal O=\frak p^{-1} \mathcal O^*$.

 		Now, we  consider the set $W_i$  defined by $ W_i=\mathcal O+u(i)$, for $1 \leq i \leq q-1$.  
 		Then,  we have the following properties of $W_i$: 
 		\begin{enumerate}
 			\item [(i)]  For each $1 \leq i \leq q-1$, $|W_i|=|\mathcal O|=1.$
 			
 			\item [(ii)]  For each $1 \leq i \leq q-1$ and $\xi \in W_i$, we have $\xi=x+u(i)$, for some $x \in \mathcal O$, and hence, $|\xi|=|x+u(i)|=max\{|x|, |u(i)|\}=q$, as $|x|\leq 1$ and $|u(i)|=q$.  
 			
 			\item [(iii)]  For each $i, j\in\{1, 2, \cdots, q-1\}$ and $i \neq j$, we have $|W_i \cap W_j|=0$.
 			
 			\item [(iv)] For each $1 \leq i \leq q-1$, the system $\{W_i+u(k): k \in \mathbb N_0\}$ is  a measurable partition of $K$ since the system $\{\mathcal O+ u(k): k \in \mathbb N_0\}$ is a   measurable partition of $K$, and for all $l, m \in \mathbb N_0$, $u(l)+u(m)=u(n)$, for some $n \in \mathbb N_0$ in view of Theorem \ref{Thm2.2}(c). 
 			
 			\item [(v)] The system $\{\frak p^{-j} W_i: j \in \mathbb Z, 1 \leq i \leq q-1\}$ is a measurable partition of $K$ since $\displaystyle \bigcup_{i=1}^{q-1} W_i=\frak p^{-1} \mathcal O^*$, $\displaystyle \bigcup_{j \in \mathbb Z} \frak p^{-j} \mathcal O=K,$     $\mathcal O \subset \frak P^{-1},$ and $\frak P^{-1} \backslash \mathcal O=\frak p^{-1} \mathcal O^*$. 
 		\end{enumerate}
 		Therefore,   $W= \bigcup_{i=1}^{q-1} W_i$ is a multiwavelet set of order $(q-1)$  in view of Corollary \ref{Cor3.4}. 
 		
 		Next, we consider a space $V_0$ defined by 
 		$$
 		V_0=\overline{\mbox{span}}\{\varphi(\cdot-u(k)): k \in \mathbb N_0, |\widehat{\varphi}|=\chi_{S}\},  
 		$$ 
 		where the associated scaling set $S=\mathcal O$. Then, the sequence $\{D^j V_0\}_{j \in \mathbb Z}$  is an   MRA by noting the properties of its associated scaling set (see, \cite{SMM2017}).  Here note that the scaling set $S$ has the following properties:   $S=\bigcup_{j \in \mathbb N} \frak p^{j} W$, the system $\{S+u(k): k \in \mathbb N_0\}$ is a measurable partition of $K$,  the multiwavelet set $W=\frak p^{-1} S\backslash S$,  and 
 		$$
 		|S|=\sum_{j \in \mathbb N} |\frak p^{j} W|=\sum_{j \in \mathbb N}\frac{|W|}{q^j}=\frac{|W|}{q-1}=\frac{q-1}{q-1}=1, \ \mbox{as} \ q >2.$$  
 	\end{example}

 	Next, we provide examples of Parseval frame wavelet  and multiwavelet set  for $L^2(K)$ and show that they are associated with Parseval frame MRA.  
 	
 	\begin{example} \label{Exa3.6}
 		Let $m \in \mathbb N$. Then,  the set $\frak p^{m} \mathcal O^*=\frak P^{m}   \backslash \frak P^{m+1}$  has the  following properties:
 		\begin{enumerate}
 			\item [(i)]   The system $\{\frak p^{j} (\frak p^{m} \mathcal O^*): j \in \mathbb Z\}$ is a measurable partition of $K$ since 
 			$$
 			\displaystyle \bigcup_{j \in \mathbb Z} \frak p^{-j} \mathcal O=K, \ \mbox{and} \    \mathcal O \subset \frak P^{-1}.
 			$$
 			
 			\item [(ii)]  The system 
 			$$
 			\{ \frak p^{m} \mathcal O^*+ u(k): k \in \mathbb N_0\}
 			$$ 
 			is a measurable partition of a measurable subset of $K$ since 
 			$$
 			\{\mathcal O+ u(k): k \in \mathbb N_0\}
 			$$
 			is a   measurable partition of $K$ and $\frak p^{m} \mathcal O^* \subset \frak P^m \subset  \mathcal O.$
 		\end{enumerate}
 		Therefore,  for each $m \in \mathbb N$,  the set $\frak p^{m} \mathcal O^*$ is  a Parseval  frame wavelet in  $L^2(K)$ in view of Theorem 3.3.   Next, consider a space $V_0$ defined by 
 		$$
 		V_0=\overline{\mbox{span}}\{\varphi(\cdot-u(k)): k \in \mathbb N_0, |\widehat{\varphi}|=\chi_{\frak P^{m+1}}\}. 
 		$$ 
 		Then, the sequence $\{D^j V_0\}_{j \in \mathbb Z}$ is a Parseval frame MRA by noting the properties of its associated scaling set (see, \cite{SMM2017}).  Here, the associated scaling set  is  $\frak P^{m+1}  =\bigcup_{j \in \mathbb N} \frak p^{j} (\frak p^{m} \mathcal O^*)$, and its associated   Parseval  frame wavelet set is $\frak p^{m} \mathcal O^*=\frak P^{m}  \backslash \frak P^{m+1}$. Further, note that the measure of scaling set is     $
 		|\frak P^{m+1}|=\frac{1}{q^{m+1}},$ and the system 
 		$$
 		\{\frak P^{m+1} +u(k): k \in \mathbb N_0\}
 		$$ 
 		is a measurable partition of a subset of $K$ since $\frak P^{m+1} \subset \mathcal O$, and the system $\{\mathcal O+ u(k): k \in \mathbb N_0\}$ is a   measurable partition of $K$.   
 	\end{example}
 	
 	\begin{example} \label{Exa3.7}
 		Let $m \in \mathbb N$ and consider the Example \ref{Exa3.5}. Then,  the set $ \frak p^m W=\bigcup_{i=1}^{q-1} \frak p^m W_i$ is a Parseval frame  multiwavelet of order $q-1$ in $L^2(K)$, where for each $1 \leq i \leq q-1$, the set 
 		$$
 		\frak p^m W_i=\frak P^m +\frak p^m u(i).
 		$$
 		This follows by noting that  
 		\begin{enumerate}
 			\item [(i)]  the system $\{\frak p^m W_i: 1 \leq i \leq q-1\}$ is a measurable partition of $\frak p^{m-1} \mathcal O^*$ since 
 			the system 
 			$$
 			\{W_i: 1 \leq i \leq q-1\}
 			$$
 			is a measurable partition of the set $ \frak p^{-1} \mathcal O^*$, and 
 			$$
 			|\frak p^m W_i \cap \frak p^m W_j|=q^{-m} |W_i \cap W_j|, \ \mbox{for} \  i, j \in \{1, 2, \cdots, q-1\}, 
 			$$ 
 			\item [(ii)]   the system $\{\frak p^{j} (\frak p^m W): j \in \mathbb Z\}$ is a measurable partition of $K$ since the system $\{\frak p^{j} W: j \in \mathbb Z\}$ is a measurable partition of $K$, 
 			\item [(iii)]  for each $1 \leq i \leq q-1$, the system 
 			$$
 			\{ \frak p^{m} W_i+ u(k): k \in \mathbb N_0\}
 			$$
 			is a measurable partition of a measurable subset of $K$ since $|\frak p^{m} W_i|=\frac{1}{q^m}<1$, and for $k, k' \in \mathbb N_0,  (k \neq k')$, we have  
 			\begin{align*}
 			|(\frak p^{m} W_i+u(k)) \cap (\frak p^{m} W_i+u(k'))|=&q^m |(W_i+\frak p^{-m} u(k)) \cap (W_i+\frak p^{-m} u(k'))|\\
 			=& q^m |(W_i+ u(q^m k)) \cap (W_i+ u(q^m k'))|\\
 			=&0,
 			\end{align*}
 			as the system $\{W_i+ u(k): k \in \mathbb N_0\}$ is a   measurable partition of $K$.
 		\end{enumerate}
 		Next, consider a space $V_0$ defined by 
 		$$
 		V_0=\overline{\mbox{span}}\{\varphi(\cdot-u(k)): k \in \mathbb N_0, |\widehat{\varphi}|=\chi_{S}\}, 
 		$$ 
 		where the associated scaling set is $S=\bigcup_{j \in \mathbb N}\frak p^{j} (\frak p^{m}W)$. Then, the sequence $\{D^j V_0\}_{j \in \mathbb Z}$  is a Parseval frame MRA by noting the properties of its associated scaling set (see, \cite{SMM2017}).  Here note that the scaling set $S$ has the following properties:   $\frak p^{-1} S\backslash S=\frak p^{m}W$,    $|S|=\frac{1}{q^m}$  and   $\{S +u(k)\}_{k \in \mathbb N_0}$ is a measurable partition of a subset of $K$ since $S \subset \bigcup_{j \in \mathbb N}\frak p^{j} W =\mathcal O$. 	
 		
 	\end{example}
 	\section{Super-wavelet of length $n$ for local fields}

 	Balan in \cite{B1999}, and Han and Larson in   \cite{HL2005}   introduced the notion of super-wavelets that have applications  in many areas  including  signal processing, data compression and image analysis.  The following   definition  of super-wavelets for local fields is  an analogue of Euclidean case:
 	
 	\begin{defn}\label{Def4.1}
 		Suppose that $\Theta=(\eta_1, \eta_2, ...,  \eta_n)$, where for each $i \in \{1, 2, \cdots, n\}$, $\eta_i$ is a Parseval frame wavelet  for $L^2(K)$. We
 		call the $n$-tuple $\Theta$ a \textit{super-wavelet of length} $n$ if
 		$$
 		\mathfrak B (\Theta):=\left\{\bigoplus_{i=1}^n D^jT^k \eta_i\equiv D^jT^k \eta_1 \oplus...\oplus D^jT^k \eta_n: j\in \mathbb Z, k \in \mathbb N_{0}\right\}
 		$$
 		is an orthonormal basis for $L^2(K) \oplus ... \oplus L^2(K)$ (\textit{say}, $\displaystyle \bigoplus_n L^2(K)$). Each $\eta_i$ here is called a \textit{component} of the super-wavelet. In the case when $\mathfrak B (\Theta)$ is a Parseval frame for $\displaystyle \bigoplus_n L^2(K)$, the $n$-tuple $\Theta$ is called a \textit{Parseval frame  super-wavelet.}
 	\end{defn}

 	The result given below is a characterization of a super-wavelet of length $n$ in case of local fields of positive characteristic.
 	
 	\begin{theorem}\label{Thm4.2}
 		Let $\eta_1,  ...,  \eta_n \in L^2(K).$  Then $(\eta_1, ..., \eta_n)$  is a super-wavelet of length $n$ if and only if the following equations hold:
 		\begin{itemize}
 			\item [(i)]\textit{$\sum_{j\in \mathbb Z} |\widehat{\eta}_i(\mathfrak p^j \xi)|^2=1,$ \ \ for a.e. $\xi \in K, \ i=1, ..., n,$}
 			\item [(ii)] $\sum_{j=0}^{\infty} \widehat{\eta}_i(\mathfrak p^{-j} \xi) \overline{\widehat{\eta}_i(\mathfrak p^{-j} (\xi+u(s))}=0,$ \ \ for a.e. $\xi \in K,  \ s \in  \mathbb N_0\backslash q\mathbb N_0, i=1, ..., n,$ \ and
 			\item [(iii)] $\sum_{k\in \mathbb N_0} \sum_{i=1}^n\widehat{\eta}_i(\mathfrak p^{-j} (\xi+u(k))) \overline{\widehat{\eta}_i(\xi+u(k))}=\delta_{j, 0},$ \ \ for a.e. $\xi \in K, j \in  \mathbb N_0.$
 		\end{itemize}
 	\end{theorem}

 	\begin{proof}
 		Suppose $(\eta_1, ..., \eta_n)$  is a {super-wavelet of length} $n$. Then, the system $\mathfrak B(\Theta)$ is an orthonormal basis for   $\displaystyle \bigoplus_n L^2(K)$. Therefore for each $1 \leq i \leq n,$ the function $\eta_i$ is a Parseval  frame wavelet   for $L^2(K)$, and hence the conditions (i) and (ii) follow  from equations $(3.1)$ and $(3.2)$. Now, condition  (iii)   follows from following descriptions: 
 		
 		Using the properties of $\{u(k): k \in \mathbb N_0\}$, the expression 
 		$$
 		< \bigoplus_{i=1}^n D^j{T}^{l}\eta_{i},  \bigoplus_{i=1}^n D^{j'} {T}^{l'}\eta_{i} >=\delta_{l, l'} \delta_{j, j'}, \ \mbox{for}\ l, l' \in \mathbb N_0; j, j' \in \mathbb Z,
 		$$
 		is equivalent to 
 		$$
 		< \bigoplus_{i=1}^n D^j{T}^{l}\eta_{i},  \bigoplus_{i=1}^n \eta_{i} >=\delta_{l, 0} \delta_{j, 0}, \ \mbox{for}\ l  \in \mathbb N_0; j \geq 0.
 		$$
 		Now, let  $j\geq 0$ and $k\in \mathbb N_{0}$. Since for each $m, k \in \mathbb N_0$, $\chi_{u(k)}(u(m))=1$,  and the system $\{\mathcal O +u(k): k \in \mathbb N_0\}$ is a measurable partition of $K$, we  have 
 		\begin{align*}
 		<\bigoplus_{i=1}^n D^j{T}^{k}\eta_{i},  \bigoplus_{i=1}^n \eta_{i}> =& \sum_{i=1}^n <D^j{T}^{k}\eta_{i}, \eta_{i}>
 		= \sum_{i=1}^n <\widehat{D^j{T}^{k}\eta}_{i}, \widehat{\eta}_{i}>, 
 		\end{align*}
 		and hence, we obtain
 		\begin{align*}
 		<\bigoplus_{i=1}^n D^j{T}^{k}\eta_{i},  \bigoplus_{i=1}^n \eta_{i}> =& \sum_{i=1}^n \int_K \widehat{D^j{T}^{k}\eta}_{i}(\xi) \overline{\widehat{\eta}_{i}(\xi)} d\xi\\
 		=&q^{-j/2}\sum_{i=1}^n \int_K \chi_{u(k)}(-\mathfrak p^j\xi)\widehat{\eta}_{i}(\mathfrak p^j\xi) \overline{\widehat{\eta}_{i}(\xi)} d\xi\\
 		=&q^{j/2}  \sum_{i=1}^n \int_{\bigcup_{m \in \mathbb N_0} \mathcal O+u(m)} \chi_{u(k)}(-\xi)\widehat{\eta}_{i}(\xi) \overline{\widehat{\eta}_{i}(\mathfrak p^{-j}\xi)} d\xi\\
 		=& {q^{j/2}   \int_{\mathcal O} \left(\sum_{i=1}^n \sum_{m \in \mathbb N_0}\widehat{\eta}_{i}(\xi+u(m)) \overline{\widehat{\eta}_{i}(\mathfrak p^{-j}(\xi+u(m))}\right) \overline{\chi_{u(k)}(\xi)} d\xi} \\
 		=&q^{j/2}  \overline{ \int_{\mathcal O} \left(\sum_{i=1}^n \sum_{m \in \mathbb N_0}\widehat{\eta}_{i}(\mathfrak p^{-j}(\xi+u(m)) \overline{\widehat{\eta}_{i}(\xi+u(m))} \right)  \chi_{u(k)}(\xi)  d\xi}.
 		\end{align*} 
 		Therefore, the result follows by comparing the above expression together with the Fourier coefficient and Fourier series of a function in $L^1 (\mathcal O)$, and  noting that      the system $\{\chi_{u(k)}\}_{k \in \mathbb N_0}$ is an orthonormal basis for $L^2(\mathcal O)$.

 		Conversely, suppose that conditions (i)-(iii) hold. Then by noting above discussion, to complete the proof it remains only to show that the system $\mathfrak B(\Theta)$ is dense in $\displaystyle \bigoplus_n L^2(K)$. The result  follows by writing the following for every $m  \in \{1, 2, \cdots, n\}$,  
 		\begin{align*}
 		\bigoplus_{i=1}^n \left(\delta_{i, m} \times g_m\right)= \sum_{(j', k') \in \mathbb Z \times \mathbb N_0} <\bigoplus_{i=1}^n \left(\delta_{i, m} \times g_m\right),   D^{j'}T^{k'} \eta_{m}>   D^{j'}T^{k'} \eta_{m}
 		\end{align*}
 		where $g_m=    D^jT^k \eta_m$. This fact is true in view of the following:  for $l=1, 2, \cdots, n$, $j \in \mathbb Z$ and $k \in \mathbb N_0$, we can write
 		\begin{align*}
 		\bigoplus_{i=1}^n D^jT^k \eta_i=&\sum_{(j', k') \in \mathbb Z \times \mathbb N_0} <\bigoplus_{i=1}^n D^jT^k \eta_i, \bigoplus_{i'=1}^n D^{j'}T^{k'} \eta_{i'}> \bigoplus_{i'=1}^n D^{j'}T^{k'} \eta_{i'}\\
 		=&\sum_{(j', k') \in \mathbb Z \times \mathbb N_0} \sum_{i=1}^n <D^jT^k \eta_i,  D^{j'}T^{k'} \eta_i> \bigoplus_{i'=1}^n D^{j'}T^{k'} \eta_{i'},
 		\end{align*}
 		and $\displaystyle D^jT^k \eta_l=\sum_{(j', k') \in \mathbb Z \times \mathbb N_0} <D^jT^k \eta_l,   D^{j'}T^{k'} \eta_{l}>   D^{j'}T^{k'} \eta_{l}$, and hence we have 
 		\begin{align*}
 		\sum_{(j', k') \in \mathbb Z \times \mathbb N_0} <D^jT^k \eta_l,   D^{j'}T^{k'} \eta_{l}>   D^{j'}T^{k'} \eta_{l'}=0
 		\end{align*}
 		for $l \neq l'$ and $l, l' \in \{1, 2, \cdots, n\}$.  
 	\end{proof}

 	The following is an easy consequence of above theorem:

 	\begin{theorem}\label{Thm4.3} 
 		Let $\eta_1,  ...,  \eta_n \in L^2(K)$ be such that $|\eta_i|=\chi_{_{W_i}},$ for $i \in \{1, 2, \cdots, n\}.$  Then $(\eta_1, ..., \eta_n)$  is a {super-wavelet of length} $n$ if and only if the following equations hold:
 		\begin{itemize}
 			\item [(a)]  for each  $i \in \{1, 2, \cdots, n\},$ the system  $\{ \mathfrak p^j W_i: j \in \mathbb Z\}$ is a measurable partition of $K$,
 			\item [(b)]  for each  $i \in \{1, 2, \cdots, n\},$ the system  $\{   W_i+u(k): k \in \mathbb N_0\}$ is a measurable partition of a subset of $K$,
 			\item [(c)]  the system $\{W_i+u(k): k \in \mathbb N_0, 1 \leq i \leq n\}$	is a measurable partition of $K$. 
 		\end{itemize}
 	\end{theorem}
 	
 	\begin{proof}
 		Suppose $(\eta_1, ..., \eta_n)$  is a super-wavelet of length $n$ such that  $|\eta_i|=\chi_{_{W_i}},$ for $i \in \{1, 2, \cdots, n\}.$ Then, for each $i \in \{1, 2, \cdots, n\}$, the function $\eta_i$ is a Parseval frame wavelet in $L^2(K)$ and the system $\mathcal B(\Theta)$ is an orthonormal basis for $\displaystyle \bigoplus_n L^2(K)$. Hence  the conditions (a) and (b) hold   in view of Parseval frame wavelet $\eta_i$ and Theorem 3.3, and also, the condition (iii) of Theorem 4.2 is satisfied, that means, for $ j \in \mathbb N_0$ 
 		\begin{align*}
 		\hspace{2cm} \delta_{j, 0}=&\sum_{k\in \mathbb N_0} \sum_{i=1}^n\widehat{\eta}_i(\mathfrak p^{-j} (\xi+u(k))) \overline{\widehat{\eta}_i(\xi+u(k))}\\
 		=&\sum_{k\in \mathbb N_0} \sum_{i=1}^n\chi_{W_i}(\mathfrak p^{-j} (\xi+u(k))) \chi_{W_i}(\xi+u(k))\\ 
 		=& \sum_{k\in \mathbb N_0} \sum_{i=1}^n\chi_{\left(\mathfrak p^{j} W_i+ u(k)\right) \cap \left( W_i+ u(k)\right)}(\xi),
 		\end{align*}
 		which is true for $j \neq 0$ since 
 		$$
 		|\left(\mathfrak p^{j} W_i+ u(k)\right) \cap \left( W_i+ u(k)\right)|=0,
 		$$ 
 		in view of conditions (a) and (b). Now, let $j=0$. Then, the expression   $$
 		\sum_{k\in \mathbb N_0} \sum_{i=1}^n\chi_{  \left( W_i+ u(k)\right)}(\xi)=1
 		$$   implies that    
 		$$
 		|\left(W_l+ u(k)\right) \cap \left( W_{l'}+ u(k')\right)|=0
 		$$
 		for $k, k' \in \mathbb N_0$; $l, l' \in \{1, 2, \cdots, n\}$ and $(l, k) \neq (l', k')$. Also, we have 
 		\begin{align*}
 		\hspace{2cm}	1=|\mathcal O|=&\int_{\mathcal O} d\xi= \int_{\mathcal O} \sum_{k\in \mathbb N_0} \sum_{i=1}^n\chi_{  \left( W_i+ u(k)\right)}(\xi) d\xi\\
 		=&   \int_K \chi_{\cup_{i=1}^n W_i} (\xi) d\xi=|\cup_{i=1}^n W_i|, 
 		\end{align*}
 		which proves  condition (c). 
 		
 		Conversely, let us assume that for each $i \in \{1, 2, \cdots, n\}$,  the function $\eta_i$ satisfies the conditions (a), (b) and (c), where   $|\eta_i|=\chi_{_{W_i}}.$  Then, $(\eta_1, ..., \eta_n)$  is a super-wavelet of length $n$. This follows by noting Theorem 3.3, Theorem 4.2 and above calculations. 		
 	\end{proof}

 	A further research in the context of super-wavelets associated with Parseval frame MRA  on local fields is needed.    Analogous to the Euclidean case, one can define the notion of super-wavelets associated with Parseval frame MRA  on local fields as follows: 
 	
 	\begin{defn}\label{Def4.4}
 		A super-wavelet $(\eta_1, ..., \eta_n)$ is said to be  an \textit{MRA super-wavelet} if for each $i=1, 2, \cdots, n$,  $\eta_i$ is a Parseval frame wavelet  associated with Parseval frame MRA.
 	\end{defn}
 	
 	The above definition is motivated by the Euclidean case in which the following result plays an important role that can be derived analogous to \cite[Proposition 5.16]{HL2005}:
 	
 	\begin{theorem} \label{Thm4.5}
 		Suppose that $V_0 \subset (D\oplus D) V_0$ and the system 
 		$$
 		\{T^k f \oplus T^k g: k \in \mathbb N_0\}
 		$$
 		is an orthonormal basis for $V_0$, where $f, g \in L^2(K)$.  Then$,$ $\bigcup_{j \in \mathbb Z} (D^j\oplus D^j) V_0$ is not dense in $L^2(K)\oplus L^2(K).$ 
 	\end{theorem}

 	Next, we provide    examples of super-wavelet of length $n$, and   Parseval frame super-wavelet of length $n$    for the local field having positive characteristics:

 	\begin{example}\label{Exa4.6}
 		Consider the functions $\eta_i$, for $i \in \{1, 2, \cdots, n-1\}$ whose Fourier transforms are defined by 
 		$$
 		|\widehat{\eta}_i|=\chi_{\frak p^{i-1} \mathcal O^*} =\chi_{\frak  p^{i-1} (\mathcal O \backslash \frak p \mathcal O)},
 		$$
 		where $n \geq 2$.   Then,  the collection $\{\eta_1, \eta_2, \cdots,  \eta_{n-1}\}$ has  the following properties:
 		\begin{enumerate}
 			\item [(i)] for each $i \in \{1, 2, \cdots, n-1\}$, the system  
 			$$
 			\{ \mathfrak p^j (\frak p^{i-1} \mathcal O^*): j \in \mathbb Z\}
 			$$ 
 			is a measurable partition of $K$ as $\{ \mathfrak p^j   \mathcal O^* : j \in \mathbb Z\}$ is a measurable partition of $K$,  
 			
 			\item [(ii)] for each $i \in \{1, 2, \cdots, n-1\}$, the system  
 			$$
 			\{   \frak p^{i-1} \mathcal O^*+u(k): k  \in \mathbb N_0\}
 			$$
 			is a measurable partition of a subset of  $K$ as $\{\mathcal O +u(k) : k \in \mathbb N_0\}$ is a measurable partition of $K$, and $\frak p^j \mathcal O^* \subset \mathcal O,$ where $j \in \mathbb N_0,$  
 			
 			\item [(iii)]  for $i, j \in \{1, 2, \cdots, n-1\}$ and $k, l \in \mathbb N_0$, we have 
 			$$
 			|(\frak p^{i-1} \mathcal O^*+u(k)) \cap    (\frak p^{j-1} \mathcal O^*+u(l)) |=0,  \ \mbox{for} \ (i, l) \neq (j, k), 	
 			$$
 			since $\frak p^{i-1} \mathcal O^*,  \frak p^{j-1} \mathcal O^* \subset \mathcal O$, the system $\{\mathcal O +u(k): k \in \mathbb N_0\}$ is a measurable partition of $K$, and 
 			$$
 			|\frak p^{i-1} \mathcal O^* \cap  \frak p^{j-1} \mathcal O^*|=0, \  \mbox{for} \  i \neq j,
 			$$
 			as $|x|=\frac{1}{q^{i-1}}$ and $|y|=\frac{1}{q^{j-1}}$, for $x \in \frak p^{i-1} \mathcal O^*$, and $y \in \frak p^{j-1} \mathcal O^*$.  
 		\end{enumerate}
 		Next, let us assume   the set $\frak S \subset K$ be such that $\{\frak {p}^j\frak{S}: j \in \mathbb Z\}$ is a measurable partition of $K,$ and there is a bijective map from $\frak S$ to $\frak {p}^{n-2} \mathcal O$ defined by 
 		$$
 		\xi \longmapsto \xi +u(l),
 		$$
 		for every $\xi \in \frak S$ and for some  $l \in \mathbb N_0$.  
 		Here, the existence of such set follows by noting Theorem 1 of Dai, Larson, and Speegle \cite{DLS1997}. Then, $\Theta=(\eta_1, \eta_2, \cdots,  \eta_n)$ is a super-wavelet of length $n$,  where 
 		$$
 		|\widehat{\eta}_n|=\chi_{\frak S}.
 		$$ 
 		This follows by noting Theorem 4.3 and observing that     
 		for each $i \in \{1, 2, \cdots, n\}$, the function $\eta_i$ is a Parseval frame wavelet in $L^2(K)$, and the set 
 		$$
 		\{\frak p^{i-1} \mathcal O^*+u(k), \frak S+u(l): 1 \leq i \leq n-1; k, l \in \mathbb N_0 \}
 		$$
 		is a measurable partition of $K$. 
 		
 		Further analogous to Euclidean case \cite{DL2011, ZX2012}, if we assume conditions (a) \& (b) of Theorem 4.3 and replace the condition (c) of Theorem 4.3 by \lq\lq \textit{the system 
 			$$
 			\{W_i+u(k): k \in \mathbb N_0, 1 \leq i \leq n\}
 			$$
 			is a measurable partition of a measurable set of $K$}\rq\rq, then we call   $(\eta_1, ..., \eta_n)$ as a \textit{Parseval frame super-wavelets}.    Following is an example of Parseval frame super-wavelet of length $n$. 
 	\end{example}
 	
 	\begin{example}\label{Exa4.7}
 		Consider the functions $\eta_i$, for $i \in \{1, 2, \cdots, n\}$ whose Fourier transforms are defined by 
 		$ 
 		|\widehat{\eta}_i|=\chi_{\frak p^{i} \mathcal O^*}.
 		$  
 		Then,  $\Theta=(\eta_1, \eta_2, \cdots, \eta_n)$ is a Parseval frame super-wavelet of length $n$.  In addition, this is associated with Parseval frame MRA. For more details, see the above Example 4.6, and notice that the system 
 		$$
 		\{\frak p^i \mathcal O^*+u(k): k\in \mathbb N_0, i \in \{1, 2, \cdots, n\}\}
 		$$
 		is a measurable partition of a subset of  $K$. 
 	\end{example}

 	
 	\section{Decomposable Parseval  frame wavelets    for local fields}

 	In this section we study the extendable and decomposable Parseval  frame wavelets and their properties with respect to the local field $K$ of positive characteristics while the same was  studied by   many authors   for the  case of Euclidean space \cite{DDG2004,GH2005,HL2005}.   A Parseval frame wavelet  $\eta$  is said to be  an $n$-\textit{decomposable} $(n > 1)$ if $\eta$ is equivalent to a Parseval frame super-wavelet   of length $n$. By an \textit{equivalent}  Parseval frame  super-wavelets   $(\eta_1, . . . , \eta_m)$ and $(\mu_1, . . . , \mu_n)$, we mean that  there is a unitary operator 
 	$$
 	U : \bigoplus_{i=1}^m L^2(K) \rightarrow  \bigoplus_{i=1}^n L^2(K) $$
 	such that
 	$$
 	U(D^kT^l \eta_1\oplus . . . \oplus D^kT^l \eta_m) = (D^kT^l \mu_1\oplus . . . \oplus D^kT^l \mu_n),
 	$$
 	for all $l \in \mathbb N_0, k\in \mathbb Z$.  The following  result provides a characterization of the equivalence between two Parseval frame super-wavelets:
 	
 	\begin{proposition}\label{Pro5.1}  
 		Suppose that $(\psi_1, . . . , \psi_M)$ and $(\varphi_1, . . . , \varphi_N)$ are Parseval frame super-wavelets. Then they are equivalent if and only if  for a.e. $\xi$ and $n\in \mathbb N_0$, 
 		$$
 		\sum_{j=1}^M\sum_{k \in \mathbb N_0} \widehat{\psi}_j(\frak p^{-n}(\xi+u(k)))\overline{\widehat{\psi}_j(\xi+u(k))}=\sum_{j=1}^N\sum_{k \in \mathbb N_0} \widehat{\varphi}_j(\frak p^{-n}(\xi+u(k)))\overline{\widehat{\varphi}_j(\xi+u(k))}.
 		$$
 	\end{proposition}
 	
 	\begin{proof} 
 		The result follows by noting that  $(\psi_1, . . . , \psi_M)$ and $(\varphi_1, . . . , \varphi_N)$ are Parseval frame super-wavelets if and only if
 		\begin{align*}
 		\sum_{j=1}^M <D^{-n}T^m \psi_j, T^l \psi_j>=\sum_{j=1}^N <D^{-n}T^m \varphi_j, T^l \varphi_j>,
 		\end{align*}
 		for each   $m, n, l \in \mathbb N_0$. Further, notice that for each   $m, n, l \in \mathbb N_0$, we have
 		\begin{align*}
 		\sum_{j=1}^M <D^{-n}T^m \psi_j, T^l \psi_j>=&\sum_{j=1}^M <\widehat{D^{-n}T^m \psi_j}, \widehat{T^l \psi_j}>\\
 		=& \sum_{j=1}^M q^{n/2} \int_{K} \chi_{u(l)}(\frak p^{-n} \xi) \widehat{\psi}_j(\frak p^{-n} \xi)\cdot\chi_{u(l)}(\xi)  \overline{\widehat{\psi}_j}(\xi) d\xi\end{align*}
 		since $\widehat{T^l \psi}(\xi)=\overline{\chi_{u(l)}(\xi)} \widehat{\psi}(\xi)$, and  $\widehat{D^{-n} \psi}(\xi)=q^{n/2}\widehat{\psi}(\frak p^{-n}\xi)$. As the collection $\{\mathcal O +u(k):k \in \mathbb N_0\}$ is a measurable partition of $K$, and for each $r, s \in \mathbb N_0$,  
 		$$
 		u(r q^s)=\frak p^{-s} u(r), \ \mbox{and} \   \ \chi_{u(r)}(u(s))=1,
 		$$
 		therefore we can write above expression as follows:
 		\begin{align*}
 		\sum_{j=1}^M <D^{-n}T^m \psi_j, & T^l \psi_j>\\
 		=&  q^{n/2} \int_{\mathcal O}\chi_{u(l)}(\xi)\left( \chi_{u(l)}(\frak p^{-n} \xi)\sum_{j=1}^M \sum_{k \in \mathbb N_0}  \widehat{\psi}_j(\frak p^{-n}(\xi+u(k)))\overline{\widehat{\psi}_j(\xi+u(k))}\right) d\xi.
 		\end{align*}
 		Similarly, we can write the above expression for $\varphi_i$, and the  result follows by noting that  the collection $\{\chi_{u(k)}(\xi): \xi \in \mathcal O, k \in \mathbb N_0\}$ is an orthonormal basis for $L^2(\mathcal O)$.
 	\end{proof}
 	
 	The following result  gives a necessary condition  for decomposable Parseval frame wavelets:
 	
 	\begin{proposition}\label{Prop5.2}
 		If $\psi$ is a $m$-decomposable Parseval frame wavelet$,$ then
 		$$
 		\int_{\mathcal O} \frac{\sum_{k \in \mathbb N_0} \left|\widehat{\psi}(\xi+u(k))\right|^2}{|\xi|} d\xi \geq m \frac{q-1}{q}.
 		$$
 	\end{proposition}
 	
 	\begin{proof}   
 		Suppose $\psi$ is decomposable into Parseval frame wavelets  $f_1, \cdots, f_m$, and $$
 		I=  \int_{\mathcal O} \frac{\sum_{k \in \mathbb N_0}  \left|\widehat{\psi}(\xi+u(k))\right|^2}{|\xi|} d\xi.
 		$$
 		Then we have
 		$$
 		\sum_{k \in \mathbb N_0} \left|\widehat{\psi}(\xi+u(k))\right|^2=\sum_{i=1}^m\sum_{k \in \mathbb N_0} \left|\widehat{f_i}(\xi+u(k))\right|^2,
 		$$
 		and hence by applying integrals on both sides, we have
 		\begin{align*}
 		I=    \sum_{i=1}^m\sum_{k \in \mathbb N_0} \int_{\mathcal O} \frac{\left|\widehat{f_i}(\xi+u(k))\right|^2}{|\xi|} d\xi
 		=   \sum_{i=1}^m\sum_{k \in \mathbb N_0} \int_{\mathcal O+u(k)} \frac{\left|\widehat{f_i}(\xi)\right|^2}{|\xi-u(k)|} d\xi.
 		\end{align*}
 		Notice  that for each $k \in \mathbb N_0$,  we have   
 		$$
 		|\xi-u(k)| \leq |\xi|,
 		$$ 
 		where $\xi \in \mathcal O +u(k)$, which follows by observing Theorem 2.2 along with  the fact  
 		$$
 		|\xi|=|\eta+u(k)|=max\{|\eta|, |u(k)|\}=|u(k)|,
 		$$
 		since $|\eta|< |u(k)|$, for $\eta \in \mathcal O$ and $k \geq 1$. Therefore, we have
 		\begin{align*}
 		I \geq &  \sum_{i=1}^m\sum_{k \in \mathbb N_0} \int_{\mathcal O+u(k)} \frac{\left|\widehat{f_i}(\xi)\right|^2}{|\xi|} d\xi=\sum_{i=1}^m  \int_{K} \frac{\left|\widehat{f_i}(\xi)\right|^2}{|\xi|} d\xi
 		\end{align*}
 		since $\{\mathcal O+ u(k): k\in \mathbb N_0\}$ is a measurable partition of $L^2(K)$. Further, notice that 
 		$$
 		\left\{p^{j}(\mathcal O \backslash \frak{p} \mathcal O): j \in \mathbb Z\right\}
 		$$
 		is a measurable partition of $K$ since $\bigcup_{j \in \mathbb Z} \frak p^{j}\mathcal O=K, \   a.e.$ and $\frak{p} \mathcal O \subset \mathcal O$. Therefore, we have
 		\begin{align*}
 		I \geq &   \sum_{i=1}^m  \int_{K} \frac{\left|\widehat{f_i}(\xi)\right|^2}{|\xi|} d\xi=\sum_{i=1}^m \sum_{j \in \mathbb Z} \int_{p^{j}(\mathcal O \backslash \frak{p} \mathcal O)} \frac{\left|\widehat{f_i}(\xi)\right|^2}{|\xi|} d\xi\\
 		= &  \sum_{i=1}^m \int_{\mathcal O \backslash \frak{p} \mathcal O} \frac{ \sum_{j \in \mathbb Z}\left|\widehat{f_i}(p^{j} \xi)\right|^2}{|\xi|} d\xi
 		=   \sum_{i=1}^m \int_{ \mathcal O \backslash \frak{p} \mathcal O} \frac {1}{|\xi|} d\xi
 		\end{align*}
 		since for each $i \in \{1, 2, \cdots, m\}$, we have 
 		$$
 		\sum_{j \in \mathbb Z}\left|\widehat{f_i}(p^{j} \xi)\right|^2=1, \   a.e. \ \xi
 		$$
 		because $f_i$ is a Parseval frame wavelet. Thus, the result follows by noting that
 		$$
 		\int_{ \mathcal O \backslash \frak{p} \mathcal O} \frac {1}{|\xi|} d\xi=|\mathcal O \backslash \frak{p} \mathcal O|=\frac{q-1}{q}
 		$$
 		as the Haar measure on $K^{\ast}$ is given by ${d\xi}/{|\xi|}$ and $\mathcal O \backslash \frak{p} \mathcal O$ is a group of units in $K^\ast$.
 	\end{proof}
 	
 	\begin{example}\label{Exa5.3}
 		Let $\mathcal O^*=\mathcal O \backslash \frak {p} \mathcal O$ and $|\widehat{\psi}|=\chi_{\mathcal O^*}$. Then,  $\psi$ is a Parseval frame wavelet in view of Theorem 3.3. Moreover, when $\xi \in \mathcal O$, $\sum_{k \in \mathbb N_0} \left|\widehat{\psi}(\xi+u(k))\right|^2=\chi_{\mathcal O^*}(\xi)$, hence
 		\begin{align*}
 		\int_{\mathcal O} \frac{\sum_{k \in \mathbb N_0} \left|\widehat{\psi}(\xi+u(k))\right|^2}{|\xi|} d\xi= \int_{\mathcal O^*} \frac{1}{|\xi|}d\xi=|\mathcal O^*|=\frac{q-1}{q} < m \frac{q-1}{q},
 		\end{align*}
 		for any $m \geq 2$. Therefore, $\psi$ is not decomposable in view of above Proposition  5.2.
 	\end{example}
 	
 	In view of definition of super-wavelet (or a Parseval frame super-wavelet) $(\eta_1, . . . , \eta_n)$, $\eta_i$ is necessarily a Parseval frame wavelet   for $L^2(K)$, for each $i\in \{1, . . , n\}$.  A Parseval frame wavelet   $\eta_1$ is \textit{extendable} to a super-wavelet of length $n$ (or $n$-\textit{extendable}) if there exist Parseval frame wavelets $\eta_2,  . . . , \eta_n$ such that $(\eta_1, . . . , \eta_n)$ is a super-wavelet of length $n$. The following result  gives necessary condition  for extendable super-wavelets:

 	\begin{proposition}\label{Pro5.4}
 		If $\psi$ is a   Parseval  frame wavelet  and $\psi$ is extendable to a super-wavelet of length $m+1$, where    $m\in \mathbb N$, then
 		$$
 		J\equiv \int_{\mathcal O} \frac{1-\sum_{k \in \mathbb N_0} \left|\widehat{\psi}(\xi+u(k))\right|^2}{|\xi|} d\xi \geq m \frac{q-1}{q}.
 		$$
 	\end{proposition}
 	
 	\begin{proof}
 		Suppose a Parseval frame wavelet $\psi$ is extendable to a super-wavelet of length $m+1$,  where $m\in \mathbb N$. Then, there are Parseval frame wavelets  $f_1, \cdots, f_m$ such that for almost every $\xi \in K$, we have
 		$$
 		\sum_{k \in \mathbb N_0} \left|\widehat{\psi}(\xi+u(k))\right|^2+\sum_{i=1}^m\sum_{k \in \mathbb N_0} \left|\widehat{f_i}(\xi+u(k))\right|^2=1,
 		$$
 		and hence we obtain
 		\begin{align*}
 		J=&\int_{\mathcal O} \frac{\sum_{i=1}^m\sum_{k \in \mathbb N_0} \left|\widehat{f_i}(\xi+u(k))\right|^2}{|\xi|}
 		d\xi
 		= \sum_{i=1}^m\sum_{k \in \mathbb N_0}\int_{\mathcal O} \frac{  \left|\widehat{f_i}(\xi+u(k))\right|^2}{|\xi|}
 		d\xi\\
 		\geq &   \sum_{i=1}^m  \int_{K} \frac{\left|\widehat{f_i}(\xi)\right|^2}{|\xi|} d\xi 
 		=    \sum_{i=1}^m \int_{ \mathcal O \backslash \frak{p}\mathcal O} \frac {1}{|\xi|} d\xi \\
 		= & m \frac{q-1}{q}.
 		\end{align*}
 		This computation follows same as the proof of Proposition 5.2.
 	\end{proof}

 	
 	\begin{acknowledgement}
 		The authors would like to express their sincere thanks to anonymous referee for   valuable   suggestions which definitely improved the presentation of this article.
 	\end{acknowledgement}
 	
  \noindent 	Niraj K. Shukla \hfill Saurabh Chandra Maury\\
 	Discipline of Mathematics  \hfill 	Faculty of Mathematical and Statistical Sciences\\
 	Indian Institute of Technology Indore \hfill Shri Ramswaroop Memorial University \\
 	Indore, India- 453 552 \hfill Barabanki, India \\
 	Email: o.nirajshukla@gmail.com, nirajshukla@iiti.ac.in  \hfill smaury94@gmail.com
 	  \end{document}